\providecommand{\AddToHook}[2]{}
\documentclass[hidelinks,onefignum,onetabnum]{siamart251216}
\usepackage{graphicx}
\graphicspath{{rom_figures/}}
\usepackage{amsfonts}
\usepackage{booktabs}
\usepackage{placeins}
\usepackage{algorithmicx}
\usepackage{algpseudocode}
\usepackage{amsopn}
\newsiamremark{remark}{Remark}
\crefname{remark}{Remark}{Remarks}
\usepackage{color}

\newcommand{\dv}{\mathrm{d}v}
\newcommand{\dx}{\mathrm{d}x}
\newcommand{\rd}{ \mathrm{d}}
\newcommand{\intv}[1]{\left\langle #1 \right\rangle_{\Delta v}}
\newcommand{\norm}[1]{\left\lVert #1\right\rVert}
\newcommand{\sterr}{\operatorname{err}_{\mathrm{st}}}
\usepackage{url}

\definecolor{revisionnote}{rgb}{0,0.4,0.2}

\headers{A Micro-ROM for High-Field VPFP}{S. Zhang and J. Hu}
\title{A Reduced-Order Model for the High-Field Vlasov--Poisson--Fokker--Planck System Based on a Micro--Macro Decomposition}
\author{Shiheng Zhang\thanks{Department of Applied Mathematics, University of Washington, Seattle, WA 98195, USA (\email{shzhang3@uw.edu}).}
\and Jingwei Hu\thanks{Department of Applied Mathematics, University of Washington, Seattle, WA 98195, USA (\email{hujw@uw.edu}).}}

\begin{document}
\maketitle

\begin{abstract}
We develop a projection-based reduced-order model for the high-field Vlasov--Poisson--Fokker--Planck system. Based on a micro--macro decomposition, the proposed Micro-ROM applies proper orthogonal decomposition (POD)--Galerkin reduction only to the microscopic fluctuation while retaining the macroscopic density on the full spatial grid. The field-dependent Maxwellian and collision operator introduce additional computational costs that cannot be eliminated by Galerkin projection alone. To remove the dependence on the velocity-grid resolution during online simulations, we approximate the projected Maxwellian terms using precomputed Chebyshev expansions of functions of the local electric field and exploit a separable discretization of the collision operator to assemble the reduced collision matrix exactly from precomputed blocks. A first-order implicit--explicit scheme is employed to evolve the reduced coefficients and the density. The resulting method conserves discrete mass and formally recovers a discretization of the limiting high-field model as $\varepsilon\rightarrow 0$. Its online computational cost per time step is independent of the velocity-grid resolution. Numerical experiments demonstrate the accuracy and efficiency of the proposed method across a range of parameters, including the high-field regime.
\end{abstract}

\begin{keywords}
Vlasov--Poisson--Fokker--Planck system, reduced-order modeling, POD--Galerkin projection, micro--macro decomposition, affine decomposition
\end{keywords}

\begin{MSCcodes}
65L04, 65M60, 35Q83, 35Q84
\end{MSCcodes}

\section{Introduction}
\label{sec:intro}

The Vlasov--Poisson--Fokker--Planck (VPFP) system models the evolution of charged particles subject to a self-consistent electric field and collision effects. In the high-field regime considered here, the electric-field drift and collisions act on the same fast time scale. Our goal is to develop a reduced-order model that enables efficient repeated simulations over a range of parameters. Specifically, we consider the parameter vector $\mu=(\varepsilon,\vartheta)\in\Theta$, where $\varepsilon>0$ controls the common strength of the electric-field drift and collisions, while $\vartheta$ parametrizes the initial datum and the background charge. In one spatial and one velocity dimension (1D1V), the distribution $f_\mu=f_\mu(t,x,v)$  satisfies 
\begin{equation}\label{eq:VPFP-original}
  \partial_t f_\mu + v\,\partial_x f_\mu + \frac{1}{\varepsilon}\,E_\mu\,\partial_v f_\mu
  \;=\; \frac{1}{\varepsilon}\,\bigl(\partial_{vv} f_\mu + \partial_v(v f_\mu)\bigr),
  \qquad (x,v)\in\Omega_x\times\mathbb{R},
\end{equation}
with periodic boundary conditions in $x$, decay as $|v|\rightarrow \infty$, and initial condition $f_\mu(0,x,v)=f_\vartheta^0(x,v)$. The right-hand side is the linear Fokker--Planck collision operator \cite{dougherty1964model,lenard1958plasma}. The electric field is determined by the macroscopic density $\rho_\mu:=\langle f_\mu\rangle_v$, where $\langle z\rangle_v:=\int_{\mathbb{R}}z\,\dv$, through the Poisson equation
\begin{equation}\label{eq:Poisson}
  E_\mu = -\partial_x\psi_\mu,\qquad -\partial_{xx}\psi_\mu=\rho_\mu-\eta_\vartheta,
\end{equation}
where the static background charge $\eta_\vartheta(x)$ satisfies the neutrality condition
\[
\int_{\Omega_x}\bigl(\rho_\mu(0,x)-\eta_\vartheta(x)\bigr)\,\dx=0.
\]
This condition ensures the solvability of \eqref{eq:Poisson}. We write the field solve compactly as $E_\mu=\mathcal{P}(\rho_\mu-\eta_\vartheta)$, $\mathcal{P}:=-\partial_x(-\partial_{xx})^{-1}$, where $(-\partial_{xx})^{-1}$ denotes the inverse of the periodic Laplacian restricted to functions with zero spatial average. 

Since $E_\mu$ is independent of $v$, the electric-field drift can be incorporated into the collision operator. Introducing the shifted Maxwellian
\begin{equation}\label{eq:Maxwellian_full_def}
    M_\mu(t,x,v) \;:=\; \frac{\rho_\mu(t,x)}{\sqrt{2\pi}}\,\exp\!\Bigl(-\frac{(v - E_\mu(t,x))^2}{2}\Bigr),
\end{equation}
equation \eqref{eq:VPFP-original} can then be written as
\begin{equation}\label{eq:VPFP_flux}
    \partial_t f_\mu + v\,\partial_x f_\mu \;=\; \frac{1}{\varepsilon}\,\mathcal{L}[E_\mu](f_\mu),
    \qquad
    \mathcal{L}[E_\mu](f) \;:=\; \partial_v \left( M_\mu\,\partial_v \left(\frac{f}{M_\mu}\right) \right),
\end{equation}
where $\mathcal{L}[E_\mu]$ is the generalized Fokker--Planck collision operator associated with the drift $E_\mu$. Although $M_\mu$ contains the density $\rho_\mu$, this factor cancels in $\mathcal{L}[E_\mu]$, so the operator depends on the solution only through the field. Taking the zeroth velocity moment $\langle \cdot \rangle_v$ of \eqref{eq:VPFP_flux} yields the continuity equation
\begin{equation}\label{eq:continuous_macro_equation}
    \partial_t \rho_\mu + \partial_x \langle v\,f_\mu\rangle_v \;=\; 0.
\end{equation}
In the high-field limit $\varepsilon\to0$, the collision operator in \eqref{eq:VPFP_flux} drives $f_\mu$ toward the local equilibrium $M_\mu$. Consequently, $\langle v\,f_\mu\rangle_v \rightarrow \langle v\,M_\mu\rangle_v=\rho_\mu E_\mu$. Equation \eqref{eq:continuous_macro_equation} therefore closes in the limit and reduces to \cite{nieto2001high}:
\begin{equation}\label{eq:continuous_high_field_limit}
    \partial_t\rho_\mu + \partial_x(\rho_\mu E_\mu) \;=\; 0, \qquad E_\mu = \mathcal{P}(\rho_\mu - \eta_\vartheta),
\end{equation}
which is a nonlinear transport equation driven by the self-generated electric field.

Numerical simulation of \eqref{eq:VPFP-original} presents two main challenges. The first is stiffness: explicit time integration requires time steps that resolve $\varepsilon$, whereas asymptotic-preserving (AP) schemes \cite{jin1999efficient,jin2022asymptotic} are designed to allow $\varepsilon$-independent time steps while recovering the correct macroscopic limit \eqref{eq:continuous_high_field_limit}. For the high-field VPFP system, an AP scheme was constructed in \cite{jin2011asymptotic}, followed by variational and structure-preserving extensions in \cite{carrillo2021variational,blaustein2024structure}. 

The second challenge is the cost of resolving phase space, compounded by the parameter dependence. A tensor-product discretization with $N_x$ spatial and $N_v$ velocity degrees of freedom involves $N_xN_v$ unknowns, and this computational cost must be incurred repeatedly for queries $\mu\in\Theta$. Projection-based reduced-order modeling (ROM) \cite{benner2015survey,quarteroni2011certified} is well suited to such many-query settings. The computation is split into an offline stage, performed once, in which the full-order model (FOM) is solved at selected training parameters and a low-dimensional basis is extracted, for example by proper orthogonal decomposition (POD), and an online stage, in which each new query is solved within the span of this reduced basis. This paradigm differs from the dynamical low-rank (DLR) method \cite{koch2007dynamical,lubich2014projector}, which performs model reduction online by evolving a low-rank approximation of the solution for each parameter query. AP-DLR integrators have been developed for the high-field VPFP system and the corresponding Vlasov--Amp\`ere formulation \cite{zhang2026separable,coughlin2022efficient}.

In this work, we adopt the micro--macro decomposition \cite{bennoune2008uniformly,crouseilles2011asymptotic}, in which the solution is written as the sum of its local equilibrium and a scaled nonequilibrium fluctuation,
\begin{equation}\label{eq:continuous_macro_micro_split}
    f_\mu \;=\; M_\mu + \varepsilon\,g_\mu.
\end{equation}
The pair $(\rho_\mu,g_\mu)$ is then evolved in place of $f_\mu$. Since $\mathcal{L}[E_\mu]$ annihilates $M_\mu$, the collision operator acts only on the fluctuation $g_\mu$. We retain the density at full spatial resolution and apply POD--Galerkin reduction only to the microscopic fluctuation, using a fixed phase-space basis trained over a range of physical parameters, including $\varepsilon$. The resulting Micro-ROM evolves the density together with \(r\) reduced coefficients, where \(r\) is the number of POD basis functions. This formulation also preserves the macroscopic balance: collisions do not directly contribute to the density equation. In contrast, direct POD--Galerkin reduction of the distribution function does not necessarily preserve the zero velocity moment of the collision operator, potentially introducing a spurious contribution to the density evolution that is amplified by the factor $1/\varepsilon$. As discussed in \S\ref{sec:stiff_limit} and Appendix~\ref{sec:pod_appendix}, this provides a structural motivation for the micro–macro approach, particularly in the high-field regime.

Micro--macro decomposition has also been combined with model reduction to preserve macroscopic structure. For the time-dependent radiative transfer equation, Peng et al.\ \cite{peng2024micro} developed a micro--macro reduced-basis method in which both the microscopic and macroscopic components are reduced, with an equilibrium-respecting strategy for constructing the reduced spaces. Their construction treats time and angular direction as parameters, with separate ROMs demonstrated at fixed values of $\varepsilon$. Koellermeier, Krah, and Kusch \cite{koellermeier2024macro} considered hyperbolic shallow-water moment equations and retained the conservative macroscopic variables at full order while applying POD--Galerkin or DLR reduction only to the microscopic variables. More closely related to the kinetic setting considered here, Franck et al.\ \cite{franck2024hyperbolic} developed a reduced model for the Vlasov--Poisson equation with a nonlinear Fokker--Planck collision operator, using model reduction in the velocity variable together with the discrete empirical interpolation method (DEIM) \cite{barrault2004empirical,chaturantabut2010nonlinear} for collision evaluation and corrections to enforce reduced moment conservation and preserve reduced Maxwellian equilibrium.

In the present micro--macro formulation, the field-dependent Maxwellian and collision operator introduce online computational bottlenecks that Galerkin projection alone does not remove. Direct evaluation of the projected Maxwellian terms costs $\mathcal{O}(N_xN_vr)$, while direct assembly of the reduced collision matrix costs $\mathcal{O}(N_xN_vr^2)$. These operations retain dependence on the full velocity resolution and can substantially limit the computational savings of model reduction. Our main contribution is to exploit the structure of these terms to eliminate velocity-grid operations from the online time evolution without empirical interpolation or online velocity quadrature. Specifically, at each spatial cell, we evaluate the projected Maxwellian terms using precomputed Chebyshev approximations of functions of the local field, yielding an online cost of $\mathcal{O}(N_xrK)$, where $K$ is the Chebyshev degree, typically much smaller than \(N_v\). For the collision term, we adopt the separable discretization introduced in the DLR setting in \cite{zhang2026separable}, which separates the field dependence from the velocity dependence. The sums over the velocity grid are precomputed offline for each spatial cell, allowing exact assembly of the reduced collision matrix at an online cost of $\mathcal{O}(N_xr^2)$.  With a dense reduced solve and an FFT-based Poisson solver, the resulting per time step online complexity is $\mathcal{O}(N_xr^2+N_xrK+r^3+N_x\log N_x)$, independent of $N_v$.

The rest of this paper is organized as follows. Section~\ref{sec:continuous_setup} derives the micro--macro system. Section~\ref{sec:micro_rom_semi_discrete} introduces the reduced model, identifies the online bottlenecks, and derives the continuous field maps. Section~\ref{sec:online_cost_reduction} develops the time discretization, constructs the discrete field maps, and derives the reduced collision operator, followed by the complete algorithm. Section~\ref{sec:numerical_results} presents the numerical results, and Section~\ref{sec:con} concludes the paper. 

\section{Micro--macro formulation}
\label{sec:continuous_setup}

The decomposition \eqref{eq:continuous_macro_micro_split} separates the local equilibrium from the nonequilibrium fluctuation, and turns \eqref{eq:VPFP_flux} into an evolution system for the pair $(\rho_\mu,g_\mu)$ in which the high-field limit is explicit. Since $\langle f_\mu\rangle_v=\langle M_\mu\rangle_v=\rho_\mu$, the fluctuation has zero velocity moment,
\begin{equation}\label{eq:zero_moment}
    \langle g_\mu\rangle_v \;=\; 0,
\end{equation}
and $\langle v\,f_\mu\rangle_v = \rho_\mu E_\mu + \varepsilon\,\langle v\,g_\mu\rangle_v$. Using this in the continuity equation \eqref{eq:continuous_macro_equation} gives the macro equation
\begin{equation}\label{eq:continuous_macro_update}
    \partial_t\rho_\mu \;+\; \partial_x(\rho_\mu E_\mu) \;+\; \varepsilon\,\partial_x\langle v\,g_\mu\rangle_v \;=\; 0.
\end{equation}
On the other hand, substituting \eqref{eq:continuous_macro_micro_split} into \eqref{eq:VPFP_flux} yields
\begin{equation} \label{eq:continuous_micro_update1}
\partial_tM_\mu+\varepsilon \partial_t g_\mu+v\,\partial_x M_\mu+\varepsilon v\,\partial_x g_\mu=\mathcal{L}[E_\mu](g_\mu).
\end{equation}
Since $M_\mu$ depends on time only through $(\rho_\mu,E_\mu)$, we have $\partial_t E_\mu=\mathcal{P}(\partial_t\rho_\mu)$ and $\partial_t M_\mu = \partial_\rho M_\mu\,\partial_t\rho_\mu + \partial_E M_\mu\,\mathcal{P}(\partial_t\rho_\mu)$, where $\partial_t\rho_\mu$ is represented using \eqref{eq:continuous_macro_update}. Substituting this $\partial_t M_\mu$ into \eqref{eq:continuous_micro_update1} gives the micro equation
\begin{equation}\label{eq:continuous_micro_update}
\begin{aligned}
    \partial_t g_\mu \;&+\; v\,\partial_x g_\mu
    \;-\; \partial_\rho M_\mu\,\partial_x\langle v\,g_\mu\rangle_v
    \;-\; \partial_E M_\mu\,\mathcal{P}\bigl(\partial_x\langle v\,g_\mu\rangle_v\bigr) \\
    &+\; \frac{1}{\varepsilon}\Bigl(
    v\,\partial_x M_\mu
    \;-\; \partial_\rho M_\mu\,\partial_x(\rho_\mu E_\mu)
    \;-\; \partial_E M_\mu\,\mathcal{P}\bigl(\partial_x(\rho_\mu E_\mu)\bigr)
    \Bigr) \\
    &\;=\; \frac{1}{\varepsilon}\,\mathcal{L}[E_\mu](g_\mu).
\end{aligned}
\end{equation}

Equations \eqref{eq:continuous_macro_update} and \eqref{eq:continuous_micro_update} constitute the micro--macro formulation of \eqref{eq:VPFP_flux}, from which the formal leading-order limit follows directly: as $\varepsilon\to0$, \eqref{eq:continuous_macro_update} reduces to \eqref{eq:continuous_high_field_limit}, while \eqref{eq:continuous_micro_update} yields
\begin{equation}\label{eq:continuous_first_order_correction}
    \mathcal{L}[E_\mu](g_\mu) \;=\; v\,\partial_x M_\mu \;-\; \partial_\rho M_\mu\,\partial_x(\rho_\mu E_\mu) \;-\; \partial_E M_\mu\,\mathcal{P}\bigl(\partial_x(\rho_\mu E_\mu)\bigr).
\end{equation}
The right-hand side of the above equation has zero velocity moment, since $\langle vM_\mu\rangle_v=\rho_\mu E_\mu$, $\langle\partial_\rho M_\mu\rangle_v=1$, and $\langle\partial_E M_\mu\rangle_v=0$, and therefore satisfies the solvability condition for $\mathcal{L}[E_\mu]$. Formally solving \eqref{eq:continuous_first_order_correction} under the constraint \eqref{eq:zero_moment} gives $g_\mu$ in terms of $(\rho_\mu,E_\mu)$ and provides the $O(\varepsilon)$ closure term in \eqref{eq:continuous_macro_update}. 

\section{The Micro-ROM}
\label{sec:micro_rom_semi_discrete}

In this section, we construct a reduced-order model for the micro--macro system \eqref{eq:continuous_macro_update} and \eqref{eq:continuous_micro_update}, in which we keep the density $\rho_{\mu}$ at full order and seek a reduced representation of $g_{\mu}$. For notational convenience, we suppress the dependence on the parameter $\mu$ whenever it is clear from the context.

In the offline stage, we solve the full-order problem \eqref{eq:continuous_macro_update} and \eqref{eq:continuous_micro_update} for a finite set of training parameters, $\Theta_{\mathrm{train}}:=\{\mu_j\}_{j=1}^{N_{\mathrm{train}}}\subset\Theta$ and use the resulting micro solutions to construct a reduced basis. In the online stage, the micro solution corresponding to a new query $\mu\in\Theta$ is approximated within this fixed reduced space. The collection of solutions to be approximated is therefore
\begin{equation}\label{eq:micro_manifold}
    \mathcal{G} \;:=\; \bigl\{\, g_\mu(t,\cdot,\cdot) \;:\; \mu\in\Theta,\;\; 0\le t\le T \,\bigr\},
\end{equation}
where $T$ is a prescribed final time. The approach relies on the assumption that the set of solutions encountered as the time and parameter vary can be well approximated by a single low-dimensional subspace \cite{quarteroni2011certified}.

\subsection{Galerkin reduction}
\label{sec:galerkin_reduction}

Let $\Phi=(\phi_1,\dots,\phi_r)$ be the basis functions of $(x,v)$ spanning an $r$-dimensional subspace that approximates $\mathcal{G}$, orthonormal in the inner product $\langle u,w\rangle_{xv}:=\int_{\Omega_x}\int_{\mathbb{R}} u\,w\,\dv\,\dx$. They are constructed offline from full-order snapshots of $g_{\mu}$ at the training parameters. For a new query $\mu$, we approximate $g_\mu$ in their span,
\begin{equation}\label{eq:reduced_ansatz}
    g_\mu(t,x,v) \;\approx\; \sum_{i=1}^r a_{\mu,i}(t)\,\phi_i(x,v)  \;:=\;  \Phi(x,v)\,\mathbf{a}_\mu(t),
\end{equation}
where the time and parameter dependence is carried by the coefficients $\mathbf{a}_\mu(t)=(a_{\mu,1}(t),\dots,a_{\mu,r}(t))^\top$. 

Substituting \eqref{eq:reduced_ansatz} into \eqref{eq:continuous_macro_update} yields 
\begin{equation}\label{eq:micro_rom_macro}
    \partial_t\rho + \partial_x\underbrace{(\rho E)}_{=: \mathcal J_0} + \varepsilon\,\partial_x\bigl(\langle v\,\Phi\rangle_v\,\mathbf{a}\bigr) = 0.
\end{equation}
Substituting \eqref{eq:reduced_ansatz} into \eqref{eq:continuous_micro_update} and conducting the Galerkin projection $\langle \phi_i, \cdot \rangle_{xv}$ yields
\begin{equation}\label{eq:micro_rom_micro_ode}
\begin{aligned}
    \frac{\rd a_i}{\rd t}
    \;&+\; \bigl\langle \phi_i,\, v\,\partial_x\Phi \bigr\rangle_{xv}\,\mathbf{a} +\; \frac{1}{\varepsilon}\underbrace{\langle\phi_i,\,v\,\partial_x M\rangle_{xv}}_{=: \mathcal J_{1,i}} \\[1ex]
    &-\; \frac{1}{\varepsilon}\underbrace{\Bigl\langle \phi_i,\; \partial_\rho M\,\partial_x\mathcal J_0 + \partial_E M\,\mathcal{P}(\partial_x\mathcal J_0) \Bigr\rangle_{xv}}_{=: \mathcal J_{2,i}} \\[0.5ex]
    &-\; \underbrace{\Bigl\langle \phi_i,\; \partial_\rho M\,\partial_x(\langle v\,\Phi\rangle_v\mathbf{a}) + \partial_E M\,\mathcal{P}(\partial_x(\langle v\,\Phi\rangle_v\mathbf{a})) \Bigr\rangle_{xv}}_{=: \mathcal J_{3,i}} \\[1ex]
    &=\; \frac{1}{\varepsilon}\,\bigl\langle \phi_i,\, \mathcal{L}[E]\Phi \bigr\rangle_{xv}\,\mathbf{a}, \quad i=1,\dots, r,
\end{aligned}
\end{equation}
where $E=\mathcal{P}(\rho-\eta)$, and $M=M(\rho,E)=\frac{\rho(t,x)}{\sqrt{2\pi}}\,\exp\left(-\frac{(v - E(t,x))^2}{2}\right)$.

The resulting Micro-ROM evolves the coupled system \eqref{eq:micro_rom_macro} and \eqref{eq:micro_rom_micro_ode}. While some terms, such as $\langle\phi_i, v\,\partial_x\Phi\rangle_{xv}$ and $\langle v\,\Phi\rangle_v$, depend only on the basis functions and can be computed offline, the terms involving the Maxwellian and the collision operator depend on the evolving electric field $E$. Direct evaluation of these terms would typically incur an online computational cost that depends on the velocity-grid resolution. Our goal is to exploit their structure to eliminate velocity-grid operations from the online time evolution without relying on empirical interpolation or online velocity quadrature. In the next subsection, we express the Maxwellian terms using precomputed maps evaluated at the local electric field. The collision projection is treated after discretization in \S\ref{sec:separable_fp_new}. 

\subsection{Reduction of Maxwellian terms}
\label{sec:field_maps}

Here, we express the velocity integrals in $\mathcal J_{1,i}$, $\mathcal J_{2,i}$, and $\mathcal J_{3,i}$ through precomputable one-dimensional maps of the local field value $E$, thereby eliminating the need for online velocity integration.

For $\mathcal J_{1,i}$, integration by parts in $x$ over the periodic domain gives
\begin{equation}\label{eq:transport_ibp}
\begin{aligned}
    \mathcal J_{1,i}
    &= -\,\bigl\langle v\,\partial_x\phi_i,\, M\bigr\rangle_{xv}
    = -\frac{1}{\sqrt{2\pi}}\int_{\Omega_x} \rho \int_{\mathbb{R}} v\,\partial_x\phi_i\, e^{-(v-E)^2/2}\,\dv\,\dx \\ 
    &:=\int_{\Omega_x} \rho\, h_{1,i}(x,E)\,\dx,
\end{aligned}
\end{equation}
where
\begin{equation}\label{eq:h1_kernel_def}
    h_{1,i}(x,E) =-\frac{1}{\sqrt{2\pi}}\int_{\mathbb{R}} \partial_x\phi_i \,v\; e^{-(v-E)^2/2}\,\dv.
\end{equation}

For $\mathcal J_{2,i}$, the prefactors $\partial_x\mathcal J_0$ and $\mathcal{P}(\partial_x\mathcal J_0)$ do not depend on $v$ and come outside the velocity integrals, 
\begin{equation}\label{eq:tangent_expand}
\begin{aligned}
    \mathcal J_{2,i}
    &= \int_{\Omega_x}\Bigl[\, \partial_x\mathcal J_0 \,\frac{1}{\sqrt{2\pi}}\int_{\mathbb{R}}\phi_i\,e^{-(v-E)^2/2}\,\dv \\
    &\qquad + \rho\,\mathcal{P}(\partial_x\mathcal J_0)\,\frac{1}{\sqrt{2\pi}} \int_{\mathbb{R}}\phi_i\,(v-E)\,e^{-(v-E)^2/2}\,\dv \,\Bigr]\dx\\
    &:=\int_{\Omega_x}\bigl[\, \partial_x\mathcal J_0\,h_{2,i}(x,E) + \rho\,\mathcal{P}(\partial_x\mathcal J_0)\,\partial_E h_{2,i}(x,E) \,\bigr]\dx,
\end{aligned}
\end{equation}
where
\begin{equation}\label{eq:h2_kernel_def}
    h_{2,i}(x,E) = \frac{1}{\sqrt{2\pi}}\int_{\mathbb{R}}\phi_i\,e^{-(v-E)^2/2}\,\dv,
\end{equation}
and
\begin{equation} \label{eq:h2x_kernel_def}
    \partial_E h_{2,i}(x,E) = \frac{1}{\sqrt{2\pi}}\int_{\mathbb{R}}\phi_i\, (v-E) e^{-(v-E)^2/2}\,\dv.
\end{equation}

For $\mathcal J_{3,i}$, since $\mathbf{a}$ is independent of $(x,v)$ and $\mathcal P$ is linear, we obtain
\begin{equation}\label{eq:micro_tangent_reduced}
\begin{aligned}
    \mathcal J_{3,i}
    = \sum_{j=1}^{r}a_j\int_{\Omega_x}\Bigl[\,
       \partial_x\langle v\,\phi_j\rangle_v\,h_{2,i}(x,E)  + \rho\,\mathcal P\!\bigl(\partial_x\langle v\,\phi_j\rangle_v\bigr)
       \,\partial_E h_{2,i}(x,E)\,\Bigr]\dx.
\end{aligned}
\end{equation}

Figure~\ref{fig:field_map_offline_online} summarizes this reduction. For each fixed $x$, the maps $h_{1,i}(x,\cdot)$, $h_{2,i}(x,\cdot)$, and $\partial_E h_{2,i}(x,\cdot)$ are functions of the scalar field value $E$ through the Gaussian factor. Each map can be approximated offline by a Chebyshev expansion in $E$ over an interval $[E_{\min},E_{\max}]$ covering the range of $E_\mu(t,x)$ across the training trajectories $\mu\in\Theta_{\mathrm{train}}$: 
\begin{equation}\label{eq:h_cheby_new}
    h_{\ell,i}(x,E) \;\approx\; \sum_{k=0}^{K}\hat{h}_{\ell,i,k}(x)\,\mathcal{T}_k\big(s(E)\big), \quad \ell=1,2,
    \qquad
    s(E) = \frac{2E-(E_{\min}+E_{\max})}{E_{\max}-E_{\min}}.
\end{equation}
The expansion for $\partial_E h_{2,i}$ is similar. Here $\mathcal{T}_k$ denotes the Chebyshev polynomial of the first kind of degree $k$. In the online stage, $\mathcal{T}_0,\ldots,\mathcal{T}_K$ are computed by the three-term recurrence and combined with the stored coefficients. Each map evaluation costs $\mathcal{O}(K)$ and requires no velocity integration.

\begin{figure}[htbp]
    \centering
    \includegraphics[width=\textwidth]{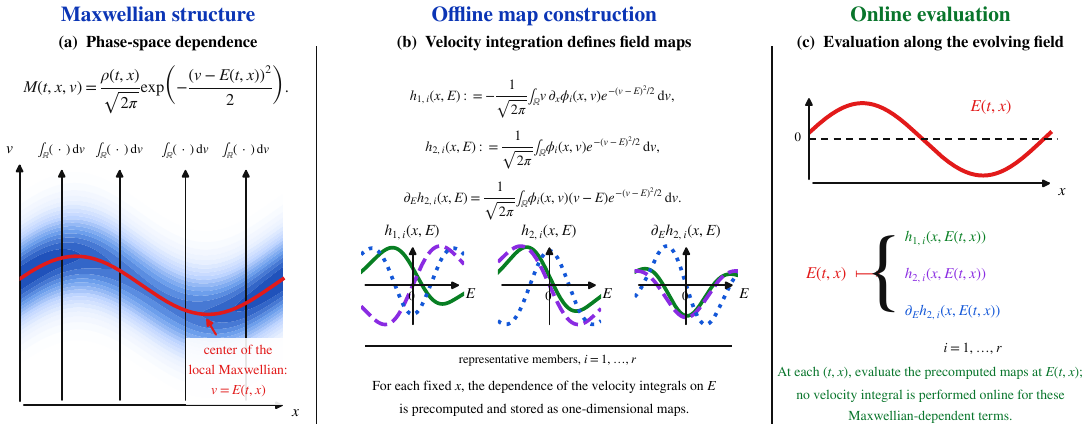}
    \caption{Offline and online treatment of the velocity integrals in $\mathcal J_{1,i}$, $\mathcal J_{2,i}$, and $\mathcal J_{3,i}$. The local Maxwellian is centered at $v=E(t,x)$. For each fixed $x$, velocity integration defines one-dimensional field maps, represented here by $h_{1,i}$, $h_{2,i}$, and $\partial_E h_{2,i}$, which are precomputed offline and evaluated online at the evolving field $E(t,x)$.}
    \label{fig:field_map_offline_online} 
\end{figure}

\begin{remark}
    In \(d_x\) spatial and \(d_v\) velocity dimensions, the velocity integrals appearing in \(\mathcal J_{1,i}\), \(\mathcal J_{2,i}\), and \(\mathcal J_{3,i}\) can be treated similarly by introducing precomputable maps of the local electric field \(E\in\mathbb R^{d_x}\). A tensor-product Chebyshev approximation of degree \(K\) in each field component requires \((K+1)^{d_x}\) coefficients and can be evaluated online without velocity integration.
\end{remark}

\section{Fully discrete Micro-ROM and online assembly}
\label{sec:online_cost_reduction}

\subsection{Time discretization}
\label{sec:imex}

We discretize the reduced system \eqref{eq:micro_rom_macro}--\eqref{eq:micro_rom_micro_ode}, first in time and then in phase space. Let $t^n=n\Delta t$, and use superscript $n$ to denote the numerical solution at time $t^n$. Set $E^n:=\mathcal{P}(\rho^n-\eta)$ and $M^n:=M(\rho^n,E^n)$. The terms $\mathcal J_0^n$ and $\mathcal J_{k,i}^n$, $k=1,2,3$, are those defined in \eqref{eq:micro_rom_macro}--\eqref{eq:micro_rom_micro_ode}, evaluated at $(\rho^n,E^n,\mathbf a^n)$. A first-order implicit--explicit (IMEX) scheme reads
\begin{subequations}\label{eq:rom_semi_discrete}
\begin{align}
    \frac{a_i^{n+1}-a_i^n}{\Delta t}
    &+ \langle\phi_i,v\,\partial_x\Phi\rangle_{xv}\,\mathbf a^n
    + \frac{1}{\varepsilon}\bigl(\mathcal J_{1,i}^n-\mathcal J_{2,i}^n\bigr)
    - \mathcal J_{3,i}^n \nonumber\\
    &= \frac{1}{\varepsilon}\,
       \langle\phi_i,\mathcal L[E^n]\Phi\rangle_{xv}\,\mathbf a^{n+1},
       \qquad i=1,\ldots,r,
    \label{eq:rom_semi_discrete_micro}\\
    \frac{\rho^{n+1}-\rho^n}{\Delta t}
    &+ \partial_x\mathcal J_0^n
    + \varepsilon\,\partial_x\bigl(\langle v\Phi\rangle_v\,\mathbf a^{n+1}\bigr)
    = 0.
    \label{eq:rom_semi_discrete_macro}
\end{align}
\end{subequations}

Note that the above discretization is a natural adaptation of the IMEX scheme used in \cite{LM10,crouseilles2011asymptotic}. When applied to the full-order model \eqref{eq:continuous_macro_update} and \eqref{eq:continuous_micro_update}, the resulting scheme is asymptotic-preserving; that is, it automatically reduces to a first-order forward Euler scheme applied to the limiting model discussed in \S\ref{sec:continuous_setup}, including the treatment of the $O(\varepsilon)$ term. For completeness, we describe the FOM in the fully discrete form in Appendix~\ref{sec:fom_imex_new}.

\subsection{Phase-space discretization}
\label{sec:semi_discrete}

We now introduce the phase-space discretization for the semi-discrete scheme \eqref{eq:rom_semi_discrete}. The periodic Poisson equation is discretized by a Fourier method, applied to the mean-zero part of the source, with the zero mode of the potential set to zero; we retain the notation $\mathcal{P}$ for the resulting grid map. Let $x_p = x_{\min} + (p-\tfrac12)\Delta x$ ($p=1,\dots,N_x$) and $v_q = v_{\min} + (q-\tfrac12)\Delta v$ ($q=1,\dots,N_v$), with $\Delta x=(x_{\max}-x_{\min})/N_x$ and $\Delta v=(v_{\max}-v_{\min})/N_v$, periodic in $x$ and zero-flux in $v$. Grid functions are written in boldface: phase-space quantities in $\mathbb{R}^N$ ($N=N_xN_v$), spatial quantities in $\mathbb{R}^{N_x}$. For phase-space grid functions $\mathbf z,\mathbf z'\in\mathbb R^N$, the velocity moment and inner product are
\begin{equation}\label{eq:discrete_inner_products}
    \bigl(\intv{\mathbf z}\bigr)_p := \Delta v\sum_{q=1}^{N_v} z_{p,q}, \qquad
    \langle\mathbf z,\mathbf z'\rangle_{\Delta x\Delta v} := \Delta x\,\Delta v\,\mathbf z^\top\mathbf z'.
\end{equation}
The induced norm is $\norm{\mathbf z}_{\Delta x\Delta v}:=\langle\mathbf z,\mathbf z\rangle_{\Delta x\Delta v}^{1/2}$. For spatial grid functions $\mathbf z\in\mathbb R^{N_x}$, we use $\norm{\mathbf z}_{\Delta x}:=(\Delta x\,\mathbf z^\top\mathbf z)^{1/2}$.

We discretize $v\,\partial_x$ using a first-order upwind scheme and represent the resulting discrete operator by the matrix $A$. With the nonnegative
splitting $v_q^+ := \max(v_q,0)$ and $v_q^- := \max(-v_q,0)$, so that
$v_q=v_q^+-v_q^-$,
\begin{equation}\label{eq:Ah_upwind_new}
    (A\mathbf z)_{p,q} \;=\; \frac{v_q^+(z_{p,q} - z_{p-1,q}) - v_q^-(z_{p+1,q} - z_{p,q})}{\Delta x},
\end{equation}
with periodic boundary conditions.

The Maxwellian is sampled on the velocity grid,
\begin{equation}\label{eq:sampled_maxwellian_new}
    M_{p,q} \;:=\; \frac{\rho_p}{\sqrt{2\pi}}\,\exp\!\left(-\frac{(v_q - E_p)^2}{2}\right).
\end{equation}
We discretize the macroscopic flux divergence $\partial_x\mathcal J_0=\langle v\partial_x M\rangle_v$ by taking the velocity moment of the same upwind operator:
\begin{equation}\label{eq:flux_moment_identity}
    \bigl(\intv{A\mathbf{M}}\bigr)_p
    = \frac{F_{p+1/2}-F_{p-1/2}}{\Delta x},
    \qquad p=1,\ldots,N_x,
\end{equation}
where
\begin{equation}\label{eq:Q1_flux}
    F_{p+1/2}
    = \Delta v\sum_{q=1}^{N_v}
    \bigl(v_q^+M_{p,q}-v_q^-M_{p+1,q}\bigr)
\end{equation}
is the first-order kinetic flux-vector-splitting (KFVS) flux \cite{MR1248055}.

The generalized Fokker--Planck operator $\mathcal{L}[E]$ is discretized as the field-dependent matrix $L[\mathbf{E}]$ in conservative flux form, with an interface Maxwellian $M_{p,q+1/2}$: at interior cells $q=2,\ldots,N_v-1$,
\begin{equation}\label{eq:fp_flux_difference_new}
\begin{aligned}
    (L[\mathbf E]\mathbf z)_{p,q}
    = \frac{1}{\Delta v^2}\Bigl[
    &M_{p,q+1/2}\left(\frac{z_{p,q+1}}{M_{p,q+1}}-\frac{z_{p,q}}{M_{p,q}}\right)\\
    &-M_{p,q-1/2}\left(\frac{z_{p,q}}{M_{p,q}}-\frac{z_{p,q-1}}{M_{p,q-1}}\right)
    \Bigr],
\end{aligned}
\end{equation}
with no-flux boundary conditions. Natural choices for $M_{p,q+1/2}$ are the geometric and arithmetic means; we use the geometric mean $M_{p,q+1/2}=\sqrt{M_{p,q}\,M_{p,q+1}}$ following \cite{zhang2026separable}.

\subsection{POD basis and fully discrete Micro-ROM}
\label{sec:pod_basis}

Assuming $T/\Delta t$ is an integer, let $N_T:=T/\Delta t$ denote the number of time steps. The FOM described in Appendix~\ref{sec:fom_imex_new} produces discrete fluctuations $\mathbf{g}_{\mu_j}^n$ at the time levels $t^n$, $n=0,\ldots,N_T$. Arranging these snapshots in columns gives the snapshot matrix
\[
    S_g
    :=
    \bigl[\,\mathbf{g}_{\mu_j}^n\,\bigr]_{\substack{j=1,\ldots,N_{\mathrm{train}}\\ n=0,\ldots,N_T}}
    \in\mathbb{R}^{N\times N_s}.
\]
Here $N_s=N_{\mathrm{train}}(N_T+1)$. The basis $\Phi=[\boldsymbol{\phi}_1,\ldots,\boldsymbol{\phi}_r]\in\mathbb{R}^{N\times r}$ consists of the $r$ leading left singular vectors of $S_g$ in the inner product \eqref{eq:discrete_inner_products}. Its span minimizes the mean-square snapshot projection error, and
\[
  \langle\boldsymbol{\phi}_i,\boldsymbol{\phi}_j\rangle_{\Delta x\Delta v}=\delta_{ij}.
\]
The basis is computed once and reused for every new query $\mu\in\Theta$; the parameter dependence is carried by the density $\boldsymbol{\rho}_\mu^n$, field $\mathbf{E}_\mu^n$, and reduced coefficients $\mathbf{a}_\mu^n$.

For a phase-space grid function $\mathbf z$ and a spatial vector $\mathbf u$, products are understood cellwise: $(\mathbf z\,\mathbf u)_{p,q}:=z_{p,q}u_p$.  
Substituting the discrete operators introduced in \S\ref{sec:semi_discrete} into the time-discrete scheme \eqref{eq:rom_semi_discrete},
\[
\begin{aligned}
    v\,\partial_x&\to A,
    &\mathcal{L}[E^n]&\to L[\mathbf{E}^n],\\
    \partial_x\mathcal J_0^n\big|_{x_p}
    &\to\bigl(\intv{A\mathbf{M}^n}\bigr)_p,
    &\langle\cdot,\cdot\rangle_{xv}
    &\to\langle\cdot,\cdot\rangle_{\Delta x\Delta v},
\end{aligned}
\]
we obtain the fully discrete Micro-ROM:
\begin{subequations}\label{eq:rom_fully_discrete}
\begin{align}
    \frac{a_i^{n+1}-a_i^n}{\Delta t}
    &+ \langle\boldsymbol{\phi}_i,A\Phi\rangle_{\Delta x\Delta v}\,\mathbf a^n
    + \frac{1}{\varepsilon}\bigl(\mathcal J_{1,i,h}^n-\mathcal J_{2,i,h}^n\bigr)
    - \mathcal J_{3,i,h}^n \nonumber\\
    &= \frac{1}{\varepsilon}
       \langle\boldsymbol{\phi}_i,L[\mathbf E^n]\Phi\rangle_{\Delta x\Delta v}
       \,\mathbf a^{n+1},
       \qquad i=1,\ldots,r,
    \label{eq:rom_fully_discrete_micro}\\
    \frac{\boldsymbol{\rho}^{n+1}-\boldsymbol{\rho}^n}{\Delta t}
    &+ \intv{A\mathbf M^n}
    + \varepsilon\,\intv{A\Phi}\,\mathbf a^{n+1}
    = \mathbf 0,
    \label{eq:rom_fully_discrete_macro}
\end{align}
\end{subequations}
where $\mathbf E^n=\mathcal P(\boldsymbol{\rho}^n-\boldsymbol{\eta})$ and $\mathbf M^n=\mathbf M(\boldsymbol{\rho}^n,\mathbf E^n)$. 

\subsection{Grid form of the Maxwellian terms}
\label{sec:compression_new}

We now express the macroscopic flux divergence $\intv{A\mathbf M^n}$ and the terms $\mathcal J_{1,i,h}^n$, $\mathcal J_{2,i,h}^n$, and $\mathcal J_{3,i,h}^n$ in \eqref{eq:rom_fully_discrete} through precomputable field maps. These are essentially the fully discrete version of the maps introduced in \S\ref{sec:field_maps}.

\paragraph{(i) Macro flux}
For the interface flux \eqref{eq:Q1_flux}, define the sampled half-flux maps
\begin{equation}\label{eq:flux_maps_def}
    H_0^\pm(E) := \frac{\Delta v}{\sqrt{2\pi}}\sum_qv_q^\pm e^{-(v_q-E)^2/2}.
\end{equation}
Then
\begin{equation}\label{eq:flux_maps_new}
    F_{p+1/2}
    = \rho_pH_0^+(E_p)-\rho_{p+1}H_0^-(E_{p+1}),
\end{equation}
and \eqref{eq:flux_moment_identity} becomes
\begin{equation}\label{eq:macro_flux_maps_new}
\begin{aligned}
    \bigl(\intv{A\mathbf{M}}\bigr)_p
    = \frac{1}{\Delta x}\Bigl[&\rho_p H_0^+(E_p)-\rho_{p+1} H_0^-(E_{p+1})\\
    &-\rho_{p-1} H_0^+(E_{p-1})+\rho_p H_0^-(E_p)\Bigr].
\end{aligned}
\end{equation}

\paragraph{(ii) Projected Maxwellian terms}
With $A$ defined in \eqref{eq:Ah_upwind_new}, we have
\[
    \langle\boldsymbol\phi_i,A\mathbf M\rangle_{\Delta x\Delta v}
    =
    \langle A^\top\boldsymbol\phi_i,\mathbf M\rangle_{\Delta x\Delta v},
\]
where $A^\top$ approximates $-v\,\partial_x$. Then
the grid counterparts of \eqref{eq:h1_kernel_def}, \eqref{eq:h2_kernel_def}, and \eqref{eq:h2x_kernel_def} are given by
\begin{equation}\label{eq:h1_grid_def} 
    H_{1,i}(x_p,E)
    := \frac{\Delta v}{\sqrt{2\pi}}\sum_q
       (A^\top\boldsymbol{\phi}_i)_{p,q}\,e^{-(v_q-E)^2/2},
\end{equation}
and
\begin{equation}\label{eq:h2_grid_def}
\begin{aligned}
    H_{2,i}(x_p,E) &:= \frac{\Delta v}{\sqrt{2\pi}}\sum_q \phi_{i,p,q}\,e^{-(v_q-E)^2/2}, \\
    \partial_E H_{2,i}(x_p,E) &= \frac{\Delta v}{\sqrt{2\pi}}\sum_q \phi_{i,p,q}\,(v_q-E)\,e^{-(v_q-E)^2/2}.
\end{aligned}
\end{equation}
Therefore, $\mathcal J_{1,i,h}$, $\mathcal J_{2,i,h}$, and $\mathcal J_{3,i,h}$ (i.e., the grid counterparts of \eqref{eq:transport_ibp}, \eqref{eq:tangent_expand}, and \eqref{eq:micro_tangent_reduced}) are given by
\begin{align}
    \mathcal J_{1,i,h}
    &= \Delta x\sum_p\rho_p\,H_{1,i}(x_p,E_p),
    \label{eq:transport_maps}\\
    \mathcal J_{2,i,h}
    &= \Delta x\sum_p\Bigl[
      \bigl(\intv{A\mathbf M}\bigr)_pH_{2,i}(x_p,E_p)\nonumber\\
    &\qquad +\rho_p\,\bigl(\mathcal P(\intv{A\mathbf M})\bigr)_p
      \,\partial_E H_{2,i}(x_p,E_p)\Bigr],
    \label{eq:tangent_grid}\\
    \mathcal J_{3,i,h}
    &= \Delta x\sum_p\Bigl[
      \bigl(\intv{A\Phi}\mathbf a\bigr)_p\,H_{2,i}(x_p,E_p)\nonumber\\
    &\qquad +\rho_p\,\bigl(\mathcal P(\intv{A\Phi}\mathbf a)\bigr)_p
      \,\partial_E H_{2,i}(x_p,E_p)\Bigr].
    \label{eq:micro_tangent_grid}
\end{align}

Each of the maps $H_{0}^{\pm}(E)$, $H_{1,i}(x_p,E)$, $H_{2,i}(x_p,E)$, and $\partial_E H_{2,i}(x_p,E)$ above is approximated using a degree-$K$ Chebyshev expansion. The resulting assembly of $\intv{A\mathbf M}$ and $\mathcal J_{1,i,h}$, $\mathcal J_{2,i,h}$, $\mathcal J_{3,i,h}$ for all $i$ costs $\mathcal O(N_xrK)$ per time step and requires no summation over the velocity grid.

\subsection{Affine structure of the collision matrix}
\label{sec:separable_fp_new}

The remaining velocity-grid operation is the reduced collision term $\langle\boldsymbol{\phi}_i,L[\mathbf E^n]\Phi\rangle_{\Delta x\Delta v}\,\mathbf a^{n+1}$ in \eqref{eq:rom_fully_discrete_micro}. Since the basis is orthonormal, this step needs to invert an $r\times r$ linear system for $\mathbf{a}^{n+1}$ with coefficient matrix $I_r-\tfrac{\Delta t}{\varepsilon}L_r^n$, where $I_r$ is the identity matrix and $(L_r^n)_{ij}:=\langle\boldsymbol{\phi}_i,L[\mathbf{E}^n]\boldsymbol{\phi}_j\rangle_{\Delta x\Delta v}$. Direct assembly of $L_r^n$ costs $\mathcal{O}(N_xN_vr^2)$. We next show that our discretization of the Fokker--Planck operator using geometric-mean interface reduces this to $\mathcal{O}(N_x r^2)$. 

Substituting $M_{p,q+1/2}=\sqrt{M_{p,q}M_{p,q+1}}$ into the flux form \eqref{eq:fp_flux_difference_new} gives, at interior cells $q=2,\ldots,N_v-1$, the three-point stencil
\begin{equation}\label{eq:fp_ratio_stencil_new}
\begin{aligned}
(L[\mathbf{E}]\mathbf{g})_{p,q}
&\;=\;
\frac{1}{\Delta v^2}
\left[\sqrt{\frac{M_{p,q}}{M_{p,q+1}}}\,g_{p,q+1} +\;
\sqrt{\frac{M_{p,q}}{M_{p,q-1}}}\,g_{p,q-1}\right.\\
&\left.\qquad-\;
\left(\sqrt{\frac{M_{p,q+1}}{M_{p,q}}} + \sqrt{\frac{M_{p,q-1}}{M_{p,q}}}\right) g_{p,q}
\right],
\end{aligned}
\end{equation}
whose field dependence enters only through the ratios of neighboring Maxwellian samples $\sqrt{M_{p,q}/M_{p,q\pm1}}$. With the sampled Maxwellian \eqref{eq:sampled_maxwellian_new} and $v_{q+1}=v_q+\Delta v$, each separates into a velocity factor and a field factor:
\begin{equation}\label{eq:alpha_beta_def_new}
\begin{aligned}
    \sqrt{\frac{M_{p,q}}{M_{p,q+1}}}
    &= \exp\!\left(\frac{(v_{q+1}-E_p)^2-(v_q-E_p)^2}{4}\right) \\
    &= \exp\!\left(\frac{\Delta v}{4}(v_q+v_{q+1})-\frac{\Delta v}{2}\,E_p\right) \\
    &= \underbrace{\exp\!\left(\frac{\Delta v}{4}(v_q+v_{q+1})\right)}_{:=\,\alpha_{q+1/2}}
    \;\cdot\;
    \underbrace{\exp\!\left(-\frac{\Delta v}{2}\,E_p\right)}_{:=\,\beta(E_p)},
\end{aligned}
\end{equation}
where $\alpha_{q+1/2}>0$ depends only on the velocity grid and $\beta(E_p)>0$ only on the local field value $E_p$. Substituting $\sqrt{M_{p,q}/M_{p,q+1}}=\alpha_{q+1/2}\beta(E_p)$ and $\sqrt{M_{p,q-1}/M_{p,q}}=\alpha_{q-1/2}\beta(E_p)$, and their reciprocals, into \eqref{eq:fp_ratio_stencil_new} gives
\begin{equation}\label{eq:fp_discrete_separable_new}
\begin{aligned}
(L[\mathbf{E}]\mathbf{g})_{p,q}
\;=\;
\frac{1}{\Delta v^2}
\Big[&\alpha_{q+1/2}\,\beta(E_p)\,g_{p,q+1} +\;
\alpha_{q-1/2}^{-1}\,\beta(E_p)^{-1}\,g_{p,q-1}\\
&-\;
\Big(\alpha_{q+1/2}^{-1}\,\beta(E_p)^{-1}
     + \alpha_{q-1/2}\,\beta(E_p)\Big) g_{p,q}
\Big],
\end{aligned}
\end{equation}
with the boundary cases at $q=1, N_v$ obtained by dropping the missing interface fluxes, that is, all terms carrying $\alpha_{1/2}^{\pm1}$ or $\alpha_{N_v+1/2}^{\pm1}$.

Define the constant velocity-space matrices $V_1, V_2\in\mathbb{R}^{N_v\times N_v}$ by their interior action at $q =2,\dots, N_v-1$,
\begin{equation}\label{eq:T01_def_new}
\begin{aligned}
    (V_1\mathbf{z})_q \;&:=\; \frac{\alpha_{q+1/2}\,z_{q+1} - \alpha_{q-1/2}\,z_q}{\Delta v^2},\\
    (V_2\mathbf{z})_q \;&:=\; \frac{\alpha_{q-1/2}^{-1}\,z_{q-1} - \alpha_{q+1/2}^{-1}\,z_q}{\Delta v^2},
\end{aligned}
\end{equation}
The collision matrix $L[\mathbf E]$ is block-diagonal in $x$, with $p$-th block given by
\begin{equation}\label{eq:L_block_diag_new}
    [L[\mathbf{E}]]_{pp} \;=\; \beta(E_p)\,V_1 + \beta(E_p)^{-1}\,V_2 \;\in\;\mathbb{R}^{N_v\times N_v}.
\end{equation}
Equivalently,
\[
    L[\mathbf E]=B[\mathbf E]\otimes V_1+B[\mathbf E]^{-1}\otimes V_2,
\]
where $B[\mathbf E]:=\operatorname{diag}\bigl(\beta(E_1),\ldots,\beta(E_{N_x})\bigr)$. Summing over the spatial cells gives the reduced entries
\begin{equation}\label{eq:Lg_affine_new}
    \langle\boldsymbol{\phi}_i, L[\mathbf{E}]\boldsymbol{\phi}_j\rangle_{\Delta x\Delta v}
    = \sum_{p=1}^{N_x}\Big[\beta(E_p)\,\langle\boldsymbol{\phi}_i, V_1\boldsymbol{\phi}_j\rangle_p + \beta(E_p)^{-1}\,\langle\boldsymbol{\phi}_i, V_2\boldsymbol{\phi}_j\rangle_p\Big],
\end{equation}
where
\[
\langle\boldsymbol{\phi}_i,V_\ell\boldsymbol{\phi}_j\rangle_p
    :=
    \Delta x\,\Delta v\sum_q
    \phi_{i,p,q}(V_\ell\boldsymbol{\phi}_j)_{p,q}, \qquad \ell=1,2.
\]
Here $V_\ell$ acts on the velocity vector in each spatial cell. These field-independent blocks can be precomputed offline. The reduced collision matrix $L_r^n$ with entry
\[
    (L_r^n)_{ij}    =\langle\boldsymbol{\phi}_i,L[\mathbf{E}^n]\boldsymbol{\phi}_j\rangle_{\Delta x\Delta v}
\]
is then assembled with coefficients $\beta(E_p^n)^{\pm1}$ at a cost of $\mathcal{O}(N_xr^2)$, independent of $N_v$.

\subsection{Algorithm, complexity, mass conservation, and asymptotic limit}
\label{sec:algorithm_new}

We summarize the fully discrete Micro-ROM in Algorithm~\ref{alg:micro_rom_step_new}.

\begin{algorithm}[htbp]
\caption{The Micro-ROM: offline precomputation and online step}
\label{alg:micro_rom_step_new}
\begin{algorithmic}[1]
\Statex \textbf{Offline}:
\State For each training parameter $\mu_j\in\Theta_{\mathrm{train}}$, run the FOM (Appendix~\ref{sec:fom_imex_new}) and collect $\mathbf{g}_{\mu_j}^n$ over the time levels $t^n$. Compute the POD basis  $\Phi\in\mathbb{R}^{N\times r}$.
\State Precompute the transport entries $\langle\boldsymbol{\phi}_i,A\boldsymbol{\phi}_j\rangle_{\Delta x\Delta v}$, the moment $\intv{A\Phi}\in\mathbb{R}^{N_x\times r}$, and the cell-$p$ blocks $\langle\boldsymbol{\phi}_i, V_\ell\boldsymbol{\phi}_j\rangle_p$ in \eqref{eq:Lg_affine_new}.
\State Construct the degree-$K$ Chebyshev approximations of $H_0^\pm$ \eqref{eq:flux_maps_def}, $H_{1,i}$ \eqref{eq:h1_grid_def}, $H_{2,i}$, $\partial_E H_{2,i}$ \eqref{eq:h2_grid_def}.
\end{algorithmic}
\begin{algorithmic}[1]
\Statex \textbf{Online}: For a new query $\mu\in\Theta$, evolve from $t^n\to t^{n+1}$: given $\boldsymbol{\rho}_\mu^n$, $\mathbf{a}_\mu^n$, $\mathbf{E}_\mu^n$, output $\boldsymbol{\rho}_\mu^{n+1}$, $\mathbf{a}_\mu^{n+1}$, $\mathbf{E}_\mu^{n+1}$.
\State Evaluate the interface fluxes $F^n_{p+1/2}$ \eqref{eq:flux_maps_new} at $(\boldsymbol{\rho}^n, \mathbf{E}^n)$ using the degree-$K$ map approximations, hence the macro flux divergence $\intv{A\mathbf{M}^n}$ \eqref{eq:macro_flux_maps_new}.
\State Form $\intv{A\Phi}\,\mathbf{a}^n$, compute $\mathcal{P}(\intv{A\mathbf{M}^n})$ and $\mathcal{P}(\intv{A\Phi}\,\mathbf{a}^n)$. 
\State Evaluate $\mathcal J_{1,i,h}^n$, $\mathcal J_{2,i,h}^n$, and $\mathcal J_{3,i,h}^n$ by \eqref{eq:transport_maps}, \eqref{eq:tangent_grid}, and \eqref{eq:micro_tangent_grid}, using the degree-$K$ map approximations and the quantities computed in online steps~1--2. 
\State Assemble the reduced collision matrix $L_r^n$ at the current field $\mathbf{E}^n$ via \eqref{eq:Lg_affine_new}.
\State Solve the $r\times r$ micro system \eqref{eq:rom_fully_discrete_micro} for $\mathbf{a}^{n+1}$, with coefficient matrix $I_r - \tfrac{\Delta t}{\varepsilon}\,L_r^n$.
\State Macro density update \eqref{eq:rom_fully_discrete_macro}: $\boldsymbol{\rho}^{n+1} \leftarrow \boldsymbol{\rho}^n - \Delta t\,\bigl(\intv{A\mathbf{M}^n} + \varepsilon\,\intv{A\Phi}\,\mathbf{a}^{n+1}\bigr)$.
\State Field update: $\mathbf{E}^{n+1} \leftarrow \mathcal{P}(\boldsymbol{\rho}^{n+1} - \boldsymbol{\eta})$.
\end{algorithmic}
\end{algorithm}

\paragraph{Cost per time step}
Online work splits as follows:
\begin{itemize}
\item Cellwise maps (steps 1, 3): $\mathcal{O}(N_x r K)$.
\item Reduced collision assembly (step 4): $\mathcal{O}(N_x r^2)$.
\item FFT-based Poisson solves (steps 2, 7), moment products and macro update (steps 2, 6), and solution of the $r\times r$ linear system (step 5): $\mathcal{O}(N_x\log N_x)+\mathcal{O}(N_xr)+\mathcal{O}(r^3)$. 
\end{itemize}
The total online complexity per time step is therefore $\mathcal{O}(N_xrK+N_xr^2+r^3+N_x\log N_x)$, independent of $N_v$. Per time step cost of the FOM, by contrast, scales as $\mathcal{O}(N_xN_v)$.

\paragraph{Mass conservation} \begin{proposition}\label{prop:micro_rom_mass_conservation}
For every $n\ge 0$, the macro update \eqref{eq:rom_fully_discrete_macro} satisfies
\begin{equation}\label{eq:micro_rom_mass_conservation}
    \Delta x\sum_p\rho_p^{n+1}\;=\;\Delta x\sum_p\rho_p^n.
\end{equation}
\end{proposition}
\begin{proof}
By periodicity and \eqref{eq:Ah_upwind_new},
\[
    \sum_p\bigl(\intv{A\mathbf z}\bigr)_p
    =\Delta v\sum_q\sum_p(A\mathbf z)_{p,q}
    =0
\]
for every phase-space grid function $\mathbf z$.
Applying this identity to $\mathbf z=\mathbf M^n$ and
$\mathbf z=\Phi\mathbf a^{n+1}$ and summing
\eqref{eq:rom_fully_discrete_macro} over $p$
proves the result.
\par
\end{proof}

\paragraph{Asymptotic limit}
With all discretization parameters and the POD basis fixed,
suppose that $\mathbf a^{n+1}$ remains bounded as $\varepsilon\to0$.
Then the microscopic contribution in
\eqref{eq:rom_fully_discrete_macro} vanishes,
and the density update formally reduces to
\begin{equation}\label{eq:micro_rom_high_field_limit}
    \boldsymbol\rho^{n+1}
    = \boldsymbol\rho^n-\Delta t\,\intv{A\mathbf M^n}, 
\end{equation}
which is a forward Euler KFVS discretization of the high-field limit \eqref{eq:continuous_high_field_limit}.


\section{Numerical results}
\label{sec:numerical_results}

In this section, we examine the accuracy of the proposed Micro-ROM across the parameter domain, compare it with the Direct-ROM, which applies model reduction directly to the original equation for $f$ (see Appendix~\ref{sec:pod_appendix}), and measure its online computational cost. We also assess the accuracy of the Chebyshev approximations of the field maps. Solution errors are measured against the FOM solutions (see Appendix~\ref{sec:fom_imex_new}).

\subsection{Test problem and error measures}
\label{sec:test_setup}

We consider the one-dimensional periodic spatial domain $[0,1]$ and the truncated velocity domain $[-12,12]$, with zero numerical flux at the velocity boundaries. Unless otherwise stated, $N_x=N_v=128$, $\Delta t=10^{-4}$, and the final time is $T=0.1$. The Chebyshev degree is chosen as $K=6$.

The parameters are the stiffness $\varepsilon$ and the initial beam velocity $v_b$, with numerical parameter domain
\[
  \Theta_{\mathrm{num}}=[10^{-6},1.5]\times[1.1,1.9]
\]
for $(\varepsilon,v_b)$. For a fixed $v_b$, the initial distribution is given by
\[
  f^0_{v_b}(x,v)
  =
  \frac{\rho^0(x)}{(1+\sqrt{T_b})\sqrt{2\pi}}
  \left[
    \exp\!\left(-\frac{v^2}{2}\right)
    +
    \exp\!\left(-\frac{(v-v_b)^2}{2T_b}\right)
  \right],
  \qquad T_b=0.05,
\]
where
\[
  \rho^0(x)=0.3+\exp\!\left(-\frac{(x-0.3)^2}{0.01}\right),
  \qquad
  \eta(x)=0.3+C_\eta\exp\!\left(-\frac{(x-0.6)^2}{0.01}\right).
\]
The constant $C_\eta>0$ is chosen so that the neutrality condition $\int_0^1(\rho^0-\eta)\,\dx=0$ holds.


We sample eight values of $\varepsilon$ log-uniformly
from each of the intervals
\[
 [10^{-1},1.5],\qquad [10^{-3},10^{-1}],\qquad [10^{-6},10^{-3}].
\]
Each sampled $\varepsilon$ is paired with an independently
sampled beam velocity $v_b$, uniform on $[1.1,1.9]$,
giving $24$ training parameter pairs. The snapshot matrix contains snapshots of $g$ at every time level along each training trajectory.

We measure trajectory errors by the relative space--time error
\[
  \sterr(u)
  =
  \left(
    \frac{\displaystyle\sum_{n=1}^{N_T}\Delta t\,
      \norm{u^n_{\mathrm{ROM}}-u^n_{\mathrm{FOM}}}_{\Delta x}^2}
    {\displaystyle\sum_{n=1}^{N_T}\Delta t\,
      \norm{u^n_{\mathrm{FOM}}}_{\Delta x}^2}
  \right)^{1/2}
\]
for $u=\rho,E$, and the analogous weighted $\Delta x\Delta v$ norm for $u=g$.

Final-time relative errors use
\[
  \frac{\norm{u^{N_T}_{\mathrm{ROM}}-u^{N_T}_{\mathrm{FOM}}}_h}
       {\norm{u^{N_T}_{\mathrm{FOM}}}_h},
\]
where $\norm{\cdot}_h$ is the spatial or phase-space discrete norm, as appropriate.

Let $P_r$ be the orthogonal projector onto the first $r$ POD modes in the discrete weighted norm $\norm{\cdot}_{\Delta x\Delta v}$. For each parameter $\mu$, define
\begin{align*}
 e_{\mathrm{proj},\mu}^{(r)}
 &=\left(
 \frac{\sum_{n=1}^{N_T}\Delta t\,
 \norm{(I-P_r)g_{\mu,\mathrm{FOM}}^n}_{\Delta x\Delta v}^2}
 {\sum_{n=1}^{N_T}\Delta t\,
 \norm{g_{\mu,\mathrm{FOM}}^n}_{\Delta x\Delta v}^2}
 \right)^{1/2}.
\end{align*}

\subsection{Reduced basis and dimension}
\label{sec:reduced_dim}

Let $\{\sigma_j\}$ be the singular values of the snapshot matrix. Figure~\ref{fig:compressibility} shows their decay and the dependence of the projection and Micro-ROM errors on $r$. The test parameters are
\[
  v_b=1.5,
  \qquad
  \varepsilon\in\{1,10^{-2},10^{-4},10^{-6}\},
  \qquad
  r\in\{4,8,12,20,40\}.
\]
None of these four parameter pairs belongs to the random training set.

\begin{figure}[htbp]
  \centering
  \includegraphics[width=\textwidth]{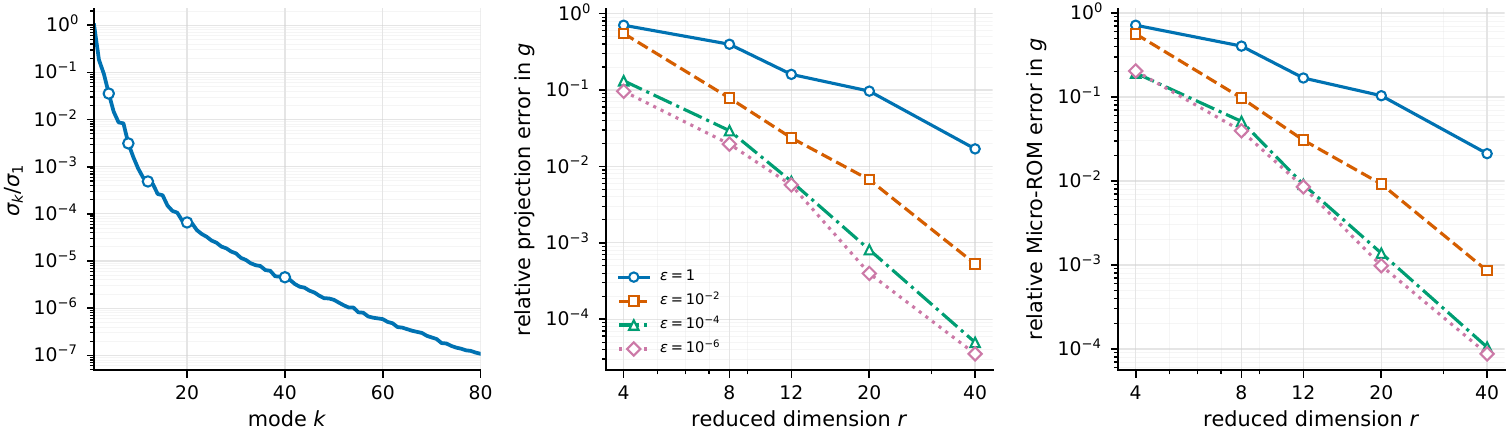}
  \caption{Singular-value decay and errors versus the reduced dimension for four values of $\varepsilon$. Left: normalized singular values of the training snapshot matrix. Center: projection error in $g$. Right: Micro-ROM space--time error in $g$. The error curves use nested leading subspaces of the same basis.}
  \label{fig:compressibility}
\end{figure}

Table~\ref{tab:rank_by_stiffness} gives the projection errors and Micro-ROM errors $\sterr(g)$ plotted in the center and right panels of Figure~\ref{fig:compressibility}. 

\begin{table}[htbp]
\centering
\small
\begin{tabular}{@{}rcccc@{}}
\toprule
& \multicolumn{4}{c}{$e_{\mathrm{proj},\varepsilon}^{(r)}$} \\
\cmidrule(lr){2-5}
$r$ & $\varepsilon=1$ & $\varepsilon=10^{-2}$ & $\varepsilon=10^{-4}$ & $\varepsilon=10^{-6}$ \\
\midrule
 4 & $7.05\times10^{-1}$ & $5.50\times10^{-1}$ & $1.32\times10^{-1}$ & $9.62\times10^{-2}$ \\
 8 & $3.97\times10^{-1}$ & $7.86\times10^{-2}$ & $2.94\times10^{-2}$ & $1.96\times10^{-2}$ \\
12 & $1.60\times10^{-1}$ & $2.36\times10^{-2}$ & $6.36\times10^{-3}$ & $5.71\times10^{-3}$ \\
20 & $9.67\times10^{-2}$ & $6.69\times10^{-3}$ & $8.12\times10^{-4}$ & $3.97\times10^{-4}$ \\
40 & $1.70\times10^{-2}$ & $5.29\times10^{-4}$ & $4.97\times10^{-5}$ & $3.52\times10^{-5}$ \\
\bottomrule
\end{tabular}
\vspace{0.7em}

\begin{tabular}{@{}rcccc@{}}
\toprule
& \multicolumn{4}{c}{{$\sterr(g)$}} \\
\cmidrule(lr){2-5}
$r$ & $\varepsilon=1$ & $\varepsilon=10^{-2}$ & $\varepsilon=10^{-4}$ & $\varepsilon=10^{-6}$ \\
\midrule
 4 & $7.15\times10^{-1}$ & $5.63\times10^{-1}$ & $1.93\times10^{-1}$ & $2.04\times10^{-1}$ \\
 8 & $4.03\times10^{-1}$ & $9.72\times10^{-2}$ & $5.12\times10^{-2}$ & $3.95\times10^{-2}$ \\
12 & $1.68\times10^{-1}$ & $3.06\times10^{-2}$ & $9.08\times10^{-3}$ & $8.51\times10^{-3}$ \\
20 & $1.03\times10^{-1}$ & $9.19\times10^{-3}$ & $1.38\times10^{-3}$ & $9.79\times10^{-4}$ \\
40 & $2.11\times10^{-2}$ & $8.53\times10^{-4}$ & $1.05\times10^{-4}$ & $8.75\times10^{-5}$ \\
\bottomrule
\end{tabular}
\caption{Projection and Micro-ROM errors in the scaled fluctuation $g$ for the four test trajectories at $v_b=1.5$.}
\label{tab:rank_by_stiffness}
\end{table}

At $r=12$, the Micro-ROM errors in $g$ are $9.08\times10^{-3}$ and $8.51\times10^{-3}$ for $\varepsilon=10^{-4}$ and $10^{-6}$, compared with $1.68\times10^{-1}$ for $\varepsilon=1$. Thus, at the same reduced dimension, the two stiff test trajectories have substantially smaller errors than the kinetic test trajectory. Unless otherwise stated, the Micro-ROM uses $r=40$ below; over these four trajectories, the largest Micro-ROM errors in $g$, $\rho$, and $E$ are $2.11\times10^{-2}$, $1.97\times10^{-3}$, and $1.10\times10^{-3}$, respectively.

\subsection{Comparison with Direct-ROM}
\label{sec:stiff_limit}

We compare the final-time density errors of the Micro-ROM with $r=40$ and the Direct-ROM of Appendix~\ref{sec:pod_appendix} with $r=40$ and $60$. We fix $v_b=1.5$, the phase-space grid, and the time step, and vary $\varepsilon$ from $10^{-1}$ to $10^{-6}$. For each $\varepsilon$, the reference is the FOM solution at the same parameter value.

\begin{figure}[htbp]
  \centering
  \includegraphics[width=0.68\textwidth]{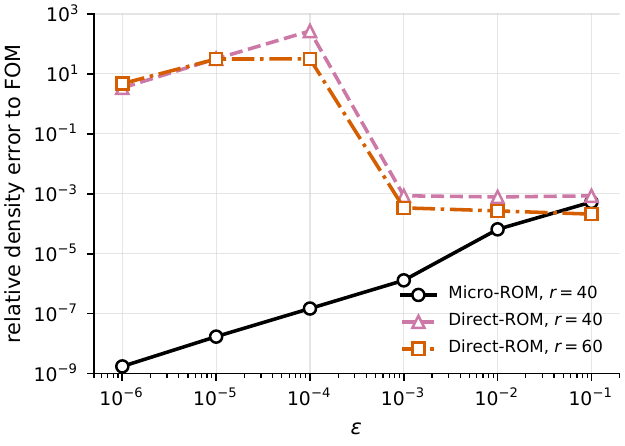}
  \caption{Final-time relative density errors against the FOM at the same $\varepsilon$, with $v_b=1.5$ and fixed discretization. The Direct-ROM is described in Appendix~\ref{sec:pod_appendix}.}
  \label{fig:stiff_limit}
\end{figure}

Figure~\ref{fig:stiff_limit} shows that the Micro-ROM error decreases from $5.29\times10^{-4}$ at $\varepsilon=10^{-1}$ to $1.74\times10^{-9}$ at $\varepsilon=10^{-6}$. For $\varepsilon=10^{-1},10^{-2},10^{-3}$, Direct-ROM approximations also have errors below $10^{-3}$. At the three smaller tested values of $\varepsilon$, however, the Direct-ROM errors exceed one; at $\varepsilon=10^{-6}$, they are $3.40$ and $4.70$ for $r=40$ and $60$, respectively.

This test highlights the importance of preserving the macroscopic balance at small $\varepsilon$. In the Direct-ROM, Galerkin projection can introduce a nonzero velocity moment of the collision term (see Appendix B), whose contribution to the density update is scaled by $1/\varepsilon$. The Micro-ROM avoids this contribution by evolving the density through a separate macroscopic equation, providing a structural explanation for its greater accuracy in the tested high-field regime.

\subsection{Accuracy at additional parameters}
\label{sec:parameter_test}

We test the same basis with $r=40$ at the parameter pairs
\[
  v_b\in\{1.2,1.4,1.6,1.8\},
  \qquad
  \varepsilon\in\{1,10^{-2},10^{-4},10^{-6}\}.
\]
All $16$ pairs are distinct from the random training set and lie within the prescribed parameter domain.

\begin{figure}[htbp]
  \centering
  \includegraphics[width=0.96\textwidth]{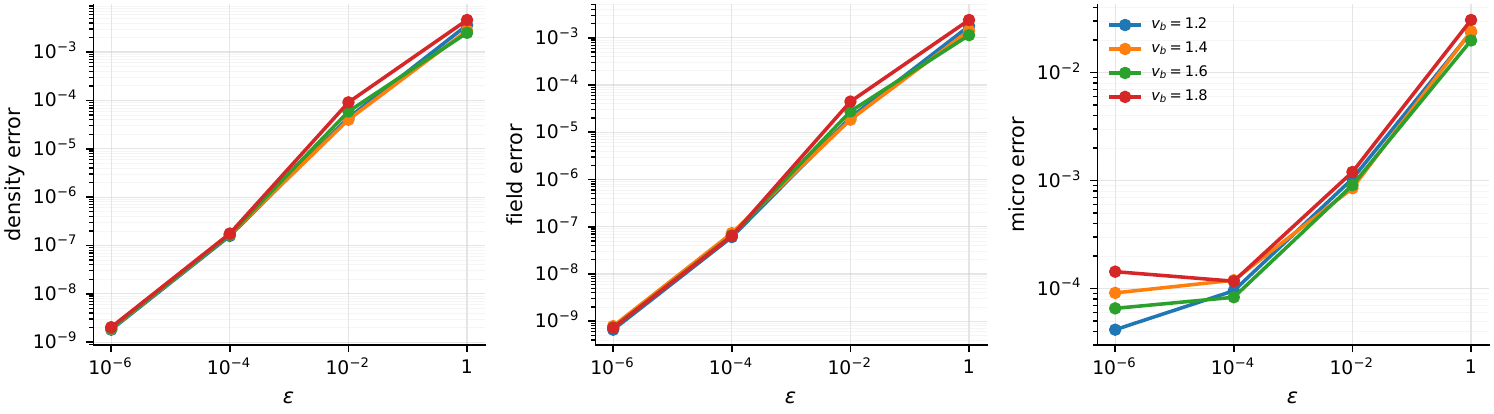}
  \caption{Relative space--time errors of the Micro-ROM with $r=40$ at $16$ test parameter pairs. The same basis is used in every panel.}
  \label{fig:parameter_test}
\end{figure}

Figure~\ref{fig:parameter_test} shows the errors at all $16$ parameter pairs, and Table~\ref{tab:parameter_maxima} reports, at each stiffness, the maximum over the four beam velocities. The largest space--time errors over all $16$ cases are $4.62\times10^{-3}$ for the density, $2.39\times10^{-3}$ for the field, and $3.07\times10^{-2}$ for the fluctuation $g$; all three occur at $(\varepsilon,v_b)=(1,1.8)$.

\begin{table}[htbp]
\centering
\small
\begin{tabular}{@{}lrrrr@{}}
\toprule
maximum over $v_b$ & $\varepsilon=1$ & $10^{-2}$ & $10^{-4}$ & $10^{-6}$ \\
\midrule
$\sterr(\rho)$ & $4.62\times10^{-3}$ & $9.14\times10^{-5}$ & $1.77\times10^{-7}$ & $2.04\times10^{-9}$ \\
$\sterr(E)$    & $2.39\times10^{-3}$ & $4.47\times10^{-5}$ & $7.31\times10^{-8}$ & $7.88\times10^{-10}$ \\
$\sterr(g)$    & $3.07\times10^{-2}$ & $1.20\times10^{-3}$ & $1.20\times10^{-4}$ & $1.43\times10^{-4}$ \\
\bottomrule
\end{tabular}
\caption{Maximum relative space--time error over $v_b\in\{1.2,1.4,1.6,1.8\}$ for each tested stiffness.}
\label{tab:parameter_maxima}
\end{table}

\subsection{Representative phase--space solutions}

Figure~\ref{fig:phase_space_solutions} compares the final FOM and Micro-ROM distributions at two values of $\varepsilon$. For each reduced dimension $r$, we retain the first $r$ modes of the same POD basis. At $(\varepsilon,v_b)=(1,1.8)$, the test pair with the largest micro space--time error among the $16$ cases above, the final-time relative discrete $L^2_{x,v}$ error in $f$ at $r=40$ is $1.88\times10^{-2}$.

For the stiff case $(\varepsilon,v_b)=(10^{-6},1.8)$ in the lower row, $r=7$ yields a final-time relative error of $6.95\times10^{-7}$ in $f$. This further confirms that significantly fewer basis functions are needed in the stiff regime.


\begin{figure}[htbp]
  \centering
  \includegraphics[width=0.98\textwidth]{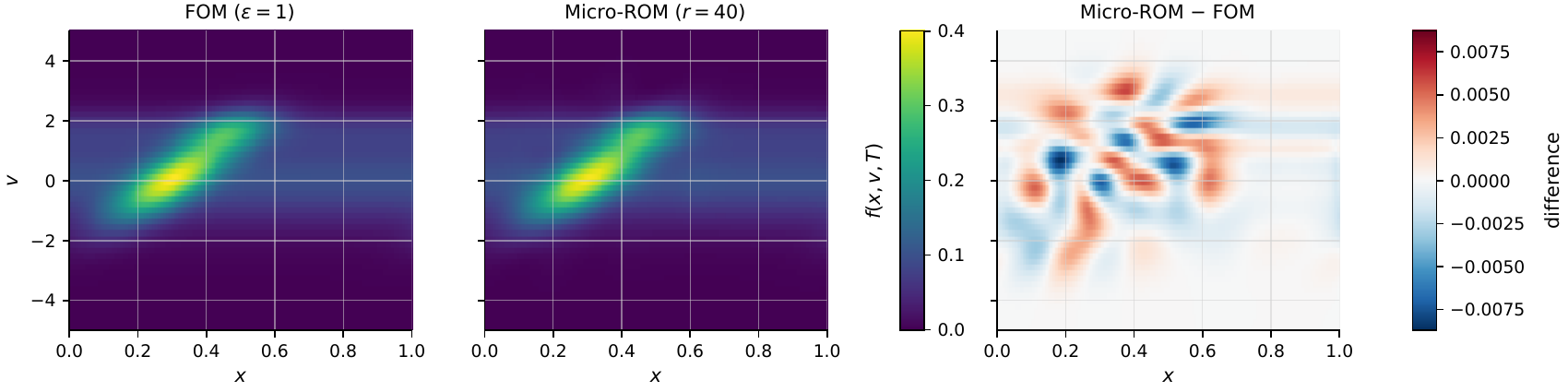}
  \par\vspace{0.6em}
  \includegraphics[width=0.98\textwidth]{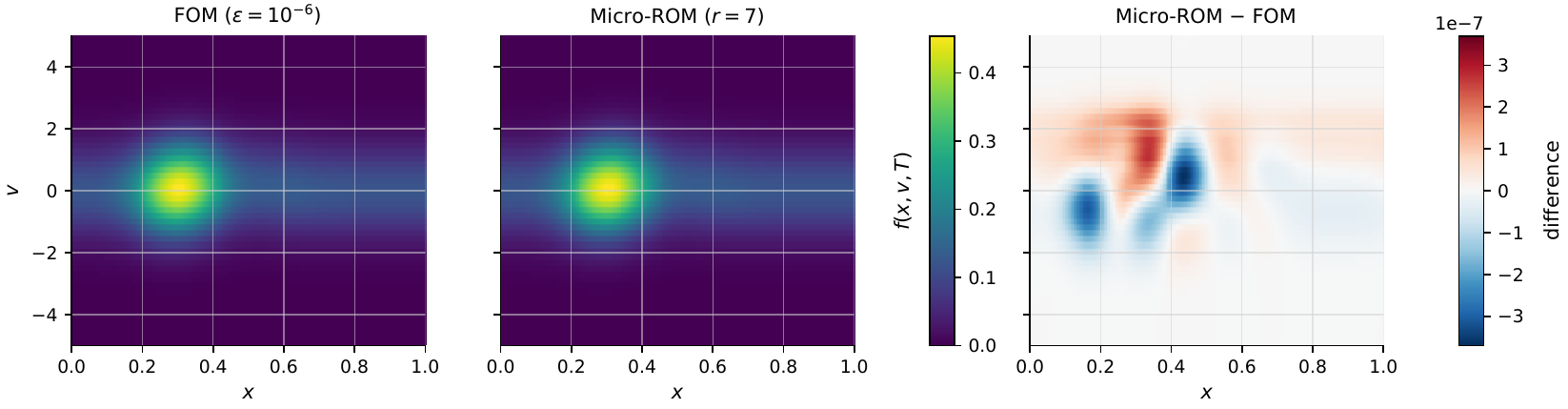}
  \caption{Final distributions at $T=0.1$ for $(\varepsilon,v_b)=(1,1.8)$ with $r=40$ (top) and $(\varepsilon,v_b)=(10^{-6},1.8)$ with $r=7$ (bottom). Each row shows the FOM, the Micro-ROM, and their difference. Within each row, the left and center panels use the same color scale. The displayed velocity interval is $[-5,5]$, while the reported relative errors are evaluated on the full computational domain $[-12,12]$.}
  \label{fig:phase_space_solutions}
\end{figure}

\subsection{Chebyshev approximation of the online field maps}
\label{sec:operator_diagnostics}

We compare the Chebyshev approximations of the field maps with their defining velocity-grid sums.

We evaluate the errors along three FOM trajectories at $v_b=1.5$, with $\varepsilon=1$, $10^{-3}$, and $10^{-6}$, using the basis with $r=40$. For a discrete map $H$ of \S\ref{sec:compression_new}, $H(\mathbf E^n)$ denotes the vector or array obtained by evaluating the map at the local field values $E_p^n$. We compute
\[
 e_H^{(K)}(\mathbf E^n):=
  \frac{\norm{\widetilde H^{(K)}(\mathbf E^n)-H(\mathbf E^n)}_2}
       {1+\norm{H(\mathbf E^n)}_2},
\]
where $\widetilde H^{(K)}$ denotes the degree-$K$ Chebyshev approximation of $H$, and $\norm{\cdot}_2$ is the Euclidean norm over all entries. Table~\ref{tab:chebyshev} reports the maximum of $e_H^{(K)}(\mathbf E^n)$ over $H_0^\pm$, $H_{1,i}$, $H_{2,i}$, and $\partial_EH_{2,i}$ and over the sampled time levels along the three trajectories.

\begin{table}[htbp]
\centering
\small
\begin{tabular}{@{}cc@{}}
\toprule
degree $K$ & maximum $e_H^{(K)}(\mathbf E^n)$ \\
\midrule
2 & $8.80\times10^{-5}$ \\
4 & $5.58\times10^{-8}$ \\
6 & $1.99\times10^{-11}$ \\
8 & $6.70\times10^{-15}$ \\
\bottomrule
\end{tabular}
\caption{Maximum error $e_H^{(K)}(\mathbf E^n)$ of the Chebyshev field maps along the three test trajectories.}
\label{tab:chebyshev}
\end{table}

\subsection{Computational time versus velocity resolution}
\label{sec:timing}

We measure the online computational time for both FOM and ROM with respect to different $N_v$. For each $N_v$, the POD basis is constructed from the same $24$ training parameter pairs in \S\ref{sec:test_setup}. The testing parameter is $(\varepsilon,v_b)=(10^{-2},1.5)$. Table~\ref{tab:timing} reports the average runtime over $15$ repeated runs to $T=0.1$ and the online speedup, defined as the ratio of the FOM runtime to the Micro-ROM runtime. Offline computation and reconstruction of $f$ are not included.

\begin{table}[htbp]
\centering
\small
\begin{tabular}{@{}rrrrrr@{}}
\toprule
$N_v$ & FOM time & \multicolumn{2}{c}{$r=20$} & \multicolumn{2}{c}{$r=40$} \\
\cmidrule(lr){3-4}\cmidrule(lr){5-6}
& (s) & ROM time (s) & speedup & ROM time (s) & speedup \\
\midrule
 32 & $0.1083$ & $0.0474$ & $2.29\times$ & $0.1043$ & $1.04\times$ \\
 64 & $0.2146$ & $0.0478$ & $4.49\times$ & $0.1065$ & $2.01\times$ \\
128 & $0.4446$ & $0.0525$ & $8.47\times$ & $0.1069$ & $4.16\times$ \\
256 & $0.9628$ & $0.0482$ & $19.96\times$ & $0.1065$ & $9.04\times$ \\
512 & $1.7795$ & $0.0481$ & $37.03\times$ & $0.1106$ & $16.09\times$ \\
\bottomrule
\end{tabular}
\caption{Average runtimes and online speedups at $N_x=128$, $\varepsilon=10^{-2}$, $v_b=1.5$, and $K=6$.}
\label{tab:timing}
\end{table}

The FOM runtime increases approximately linearly with $N_v$, while the Micro-ROM runtimes remain approximately constant at a fixed $r$. At $N_v=512$, the online speedups are $16.09\times$ for $r=40$ and $37.03\times$ for $r=20$.

All timings were obtained in double precision on an Apple M1 Max CPU.

\section{Conclusion}\label{sec:con}

We have developed a reduced-order model for the high-field VPFP system based on a micro–macro decomposition. By retaining the density at full spatial resolution and applying POD–Galerkin reduction only to the microscopic fluctuation, the method preserves the macroscopic balance while representing the evolving local equilibrium explicitly. To address the computational cost associated with the self-consistent electric field, we combine precomputed Chebyshev approximations of the projected Maxwellian terms with a separable discretization of the collision operator. Together, these constructions eliminate velocity-grid operations from the online time evolution, enabling efficient repeated simulations across a range of physical parameters. Numerical experiments demonstrate that a single reduced basis provides accurate solutions at parameter values outside the training set, including in the high-field regime. In particular, the Micro-ROM maintains accurate density evolution at small $\varepsilon$, where the uncorrected Direct-ROM exhibits substantial errors. The timing results confirm that the online computational cost remains approximately constant as the velocity resolution increases.

Future work includes extending the method to higher spatial and velocity dimensions, developing higher-order time discretizations, and establishing stability and error estimates that are uniform in $\varepsilon$. Adaptive basis construction could further improve efficiency over broader parameter domains.

\section*{Acknowledgments}
This work was partially supported by DOE grant DE-SC0023164, NSF grants DMS-2409858 and IIS-2433957, DoD MURI grant FA9550-24-1-0254, and AFOSR grant FA9550-26-1-B161.


\appendix
\section{Full-order IMEX scheme for the micro--macro system \eqref{eq:continuous_macro_update} and \eqref{eq:continuous_micro_update}}
\label{sec:fom_imex_new}

In this appendix, we describe the full-order model (FOM) used to generate the training snapshots and reference solutions in \S\ref{sec:numerical_results}. The scheme, written using the discrete operators introduced in  \S\ref{sec:semi_discrete}, is given by
\begin{subequations}\label{eq:fom_imex_struct_new}
\begin{align}
    &\frac{\mathbf{g}^{n+1} - \mathbf{g}^n}{\Delta t}
    + A\mathbf{g}^n
    - \partial_\rho \mathbf{M}^n\,\intv{A\mathbf{g}^n} - \partial_E \mathbf{M}^n\,\mathcal{P}(\intv{A\mathbf{g}^n})
    \nonumber\\
    &\qquad + \frac{1}{\varepsilon}\Bigl(A\mathbf{M}^n - \partial_\rho \mathbf{M}^n\,\intv{A\mathbf{M}^n} - \partial_E \mathbf{M}^n\,\mathcal{P}(\intv{A\mathbf{M}^n})\Bigr)
    \;=\; \frac{1}{\varepsilon}\,L[\mathbf{E}^n]\mathbf{g}^{n+1},
    \label{eq:fom_micro_struct_new}\\
     &\frac{\boldsymbol{\rho}^{n+1} - \boldsymbol{\rho}^n}{\Delta t}
    + \intv{A\mathbf{M}^n} + \varepsilon\,\intv{A\mathbf{g}^{n+1}}
    \;=\; \mathbf{0},
    \label{eq:fom_macro_struct_new}
\end{align}
\end{subequations}
with $\mathbf{E}^n=\mathcal{P}(\boldsymbol{\rho}^n-\boldsymbol{\eta})$ and $\mathbf{M}^n=\mathbf{M}(\boldsymbol{\rho}^n,\mathbf{E}^n)$. 

At each time step, the microscopic equation \eqref{eq:fom_micro_struct_new} is first solved for $\mathbf{g}^{n+1}$. Since $L[\mathbf{E}^n]$ is block-diagonal in $x$ (see \eqref{eq:L_block_diag_new}), this implicit step consists of $N_x$ independent linear systems of size $N_v\times N_v$. The density $\boldsymbol{\rho}^{n+1}$ is then obtained from the macroscopic equation \eqref{eq:fom_macro_struct_new}, followed by the field update $\mathbf{E}^{n+1}=\mathcal{P}(\boldsymbol{\rho}^{n+1}-\boldsymbol{\eta})$. The initial values are $\boldsymbol{\rho}^0=\intv{\mathbf{f}^0}$ and $\mathbf{g}^0=(\mathbf{f}^0-\mathbf{M}^0)/\varepsilon$.

\section{Direct-ROM}
\label{sec:pod_appendix}

For comparison with the Micro-ROM, we construct a reduced-order model by applying POD--Galerkin projection directly to the distribution function $\mathbf{f}$, without using the micro--macro decomposition. We refer to this approach as the Direct-ROM, which is used in \S\ref{sec:stiff_limit}. 

We first consider a first-order IMEX discretization of the original kinetic equation \eqref{eq:VPFP_flux}, using the discrete operators introduced in \S\ref{sec:semi_discrete}:
\begin{equation}\label{eq:fom_imex_kinetic}
    \frac{\mathbf{f}^{n+1}-\mathbf{f}^n}{\Delta t}
    + A\mathbf{f}^n
    \;=\; \frac{1}{\varepsilon}\,L^{n+1}\mathbf{f}^{n+1},
    \qquad L^{n+1}:=L[\mathbf{E}^{n+1}].
\end{equation}
The discrete collision operator has zero velocity moment:
\begin{equation}\label{eq:zero_collision_moment}
    \intv{L[\mathbf{E}]\,\mathbf{z}} \;=\; \mathbf{0} \qquad \text{for every } \mathbf{z}\in\mathbb{R}^N,
\end{equation}
since the flux differences \eqref{eq:fp_flux_difference_new} telescope in $v$ under the zero-flux boundary. Taking the velocity moment of \eqref{eq:fom_imex_kinetic} therefore yields the explicit density update
\begin{equation}\label{eq:fom_imex_density}
    \boldsymbol{\rho}^{n+1} \;=\; \intv{\mathbf{f}^n} - \Delta t\,\intv{A\mathbf{f}^n}.
\end{equation}
Thus, although the collision term is treated implicitly, the updated electric field can be determined before solving for $\mathbf{f}^{n+1}$.  Specifically, we first compute $\boldsymbol{\rho}^{n+1}$ from \eqref{eq:fom_imex_density}, then evaluate $\mathbf{E}^{n+1}=\mathcal{P}(\boldsymbol{\rho}^{n+1}-\boldsymbol{\eta})$ and assemble $L^{n+1}$. Finally, we solve \eqref{eq:fom_imex_kinetic} to obtain $\mathbf{f}^{n+1}$.

In the Direct-ROM, we approximate $\mathbf{f}\approx\Phi_f\,\mathbf{a}_f$, where $\Phi_f=[\boldsymbol{\phi}_{f,1},\ldots,\boldsymbol{\phi}_{f,r}]\in\mathbb{R}^{N\times r}$ is the POD basis constructed directly from snapshots of the distribution function $\mathbf{f}$. The snapshots are collected along the same training trajectories and at the same time levels as those used for the Micro-ROM. As in the Micro-ROM, the basis is orthonormal with respect to the inner product \eqref{eq:discrete_inner_products}. The initial coefficients are obtained by weighted projections $a_{f,i}^0:=\langle\boldsymbol{\phi}_{f,i},\mathbf{f}^0\rangle_{\Delta x\Delta v}$, $i=1,\ldots,r$. The density and field are recovered from the reduced solution as $\boldsymbol{\rho}=\intv{\Phi_f\mathbf{a}_f}$ and $\mathbf{E}=\mathcal{P}(\boldsymbol{\rho}-\boldsymbol{\eta})$. Define the reduced matrices by
\[
    (A_{f,r})_{ij} := \langle\boldsymbol{\phi}_{f,i},\,A\,\boldsymbol{\phi}_{f,j}\rangle_{\Delta x\Delta v},
    \qquad
    (L_{f,r}^{n+1})_{ij} := \langle\boldsymbol{\phi}_{f,i},\,L^{n+1}\boldsymbol{\phi}_{f,j}\rangle_{\Delta x\Delta v}.
\]
Galerkin projection of \eqref{eq:fom_imex_kinetic} onto the span of $\Phi_f$ yields
\begin{equation}\label{eq:pod_rom_imex_new}
    \frac{\mathbf{a}_f^{n+1}-\mathbf{a}_f^n}{\Delta t}
    + A_{f,r}\,\mathbf{a}_f^n
    \;=\; \frac{1}{\varepsilon}\,L_{f,r}^{n+1}\,\mathbf{a}_f^{n+1}.
\end{equation}

The above projected equation is generally nonlinear in $\mathbf a_f^{n+1}$. To obtain a linear reduced solve, we use the following density predictor, motivated by the full-order density update:
 \begin{equation}\label{eq:pod_rom_field_predictor_new}
    \widehat{\boldsymbol{\rho}}^{\,n+1}
    \;:=\; \intv{\Phi_f\bigl(\mathbf{a}_f^n - \Delta t\,A_{f,r}\,\mathbf{a}_f^n\bigr)}.
\end{equation}
We then set $\widehat{\mathbf{E}}^{\,n+1}=\mathcal{P}(\widehat{\boldsymbol{\rho}}^{\,n+1}-\boldsymbol{\eta})$ and assemble the reduced collision matrix $\widehat L_{f,r}^{\,n+1}$ at this predicted field. Replacing $L_{f,r}^{n+1}$ by $\widehat L_{f,r}^{\,n+1}$ in the projected equation gives a linear system for $\mathbf{a}_f^{n+1}$. The updated density is then recovered as 
$\boldsymbol{\rho}^{n+1}=\intv{\Phi_f\mathbf{a}_f^{n+1}}$. 

Unlike the full-order scheme, however, the Direct-ROM does not generally preserve the zero collision moment property under Galerkin projection. Consequently, the reduced collision term may change the recovered density, so the predicted density may not agree with the density obtained after the reduced solve. The numerical results show that this Direct-ROM exhibits substantially larger density errors than the Micro-ROM in the tested small-\(\varepsilon\) regime.

\bibliographystyle{siamplain}
\bibliography{rom_used_references}

\end{document}